\documentclass[11pt]{amsart}
\input{epsf}
\usepackage{amssymb,latexsym}
\topskip  \newtheorem{theorem}{Theorem}
\newtheorem{proposition}[theorem]{Proposition}
\newtheorem{corollary}[theorem]{Corollary}
\newtheorem{definition}[theorem]{Definition}
\newtheorem{conjecture}[theorem]{Conjecture}

\newcommand{\da}{\mbox{-}}
\newcommand{\we}{\equiv}
\newcommand{\oi}{\sim}

\renewcommand{\S}[1]{\ensuremath{\mathcal{S}_{#1}}}
\newcommand{\Sav}[2]{\ensuremath{\mathcal{S}_{#1}(#2)}}
\newcommand{\Sp}[3]{\ensuremath{\mathcal{S}_{#1}(#2)[#3]}}
\newcommand{\Spt}[4]{\ensuremath{\mathcal{S}_{#1}(#2)[#3;#4]}}
\newcommand{\Pop}{\ensuremath{\mathcal{P}}}

\newcommand{\F}{\ensuremath{\mathcal{F}}}
\newcommand{\red}{\ensuremath{\mathrm{red}}}

\begin{document}
\title{Some Wilf-equivalences for vincular patterns}

%

\author{Andrew M. Baxter}
\address{Department of Mathematics, Pennsylvania State University, State College, PA 16802 USA}
\email{baxter@math.psu.edu}

\author{Mark Shattuck}
\address{Department of Mathematics, University of Tennessee, Knoxville, TN 37996 USA}
\email{shattuck@math.utk.edu}


\maketitle
\begin{abstract}
We prove several Wilf-equivalences for vincular patterns of length $4$, some of which generalize to infinite families of vincular patterns.  We also present functional equations for the generating functions for the number of permutations of length $n$ avoiding $\sigma$ for the patterns 124-3, 134-2, 231-4, 241-3, 132-4, and 142-3.  This nearly completes the Wilf-classification of vincular patterns of length $4$.  As a corollary, these results imply Wilf-equivalences for certain barred patterns of length 5 with a single bar.
\end{abstract}

\section{Introduction}\label{sec:intro}

Vincular patterns were first introduced by Babson and Steingr\'{i}msson in 2000 in \cite{Babson2000}, and the Wilf-classification for length 3 patterns was completed soon thereafter by Claesson in \cite{Claesson2001}.  The classification for length 4 vincular patterns is still incomplete, although progress has been made in \cite{Elizalde2006, Kitaev2005, Baxter2013, Kasraoui2012}.  In this paper, we answer the remaining parts of Conjecture 17 of \cite{Baxter2012}.  This paper, combined with previous literature and computation, suggests that only two Wilf-equivalences remain for vincular patterns of length $4$, listed in the conclusion as Conjecture \ref{conj:LastWilf1}.

Let $[n]=\{1,2,\dotsc,n\}$.  For a word $w \in [n]^k$, we write  $w=w_1 w_2 \dotsm w_k$ and define the \emph{reduction} $\red(w)$ to be the word obtained by replacing the $i$-th smallest letter(s) of $w$ with $i$.  For example $\red(839183)=324132$.  If $\red(u)=\red(w)$, we say that $u$ and $w$ are \emph{order-isomorphic} and write $u\oi w$.

A permutation $\pi\in\S{n}$ \emph{contains} $\sigma\in\S{k}$ \emph{as a classical pattern} if there exists a subsequence $\pi_{i_1}\pi_{i_2}\dotsm\pi_{i_k}$ for $1\leq i_1 < i_2 < \dotsm < i_k\leq n$ such that  $\pi_{i_1}\pi_{i_2}\dotsm\pi_{i_k}\oi\sigma$.  Vincular patterns (also called ``generalized patterns'' or ``dashed patterns'') resemble classical patterns, except that some of the indices $i_j$ must be consecutive.  Formally, we can consider a vincular pattern as a pair $(\sigma, X)$ for a permutation $\sigma\in\S{k}$ and a set of adjacencies $X\subseteq [k-1]$.   The subsequence $\pi_{i_1}\pi_{i_2}\dotsm\pi_{i_k}$ for $1\leq i_1 < i_2 < \dotsm < i_k\leq n$  is a \emph{copy} (or \emph{occurrence}) of $(\sigma, X)$ if $\pi_{i_1}\pi_{i_2}\dotsm\pi_{i_k}\oi\sigma$ \emph{and} $i_{j+1}-i_{j} = 1$ for each $j\in X$.  If a copy of $(\sigma, X)$ appears in $\pi$, we say that $\pi$ \emph{contains} $(\sigma, X)$, and otherwise we say $\pi$ \emph{avoids} $(\sigma, X)$.
In practice, we write $(\sigma, X)$ as the permutation $\sigma$ with a dash between $\sigma_j$ and $\sigma_{j+1}$ if $j\not\in X$ and refer to ``the vincular pattern $\sigma$'' without explicitly writing $X$.  For example, $(1243, \{3\})$ is written $1\da2\da43$.  The permutation $162534$ has a copy of $1\da2\da43$ as witnessed by the subsequence $1253$, but the subsequence $1254$ is not a copy of $1\da2\da43$ since the 5 and 4 are not adjacent.  \emph{Classical} patterns are those of the form $(\sigma, \emptyset)$ where no adjacencies are required, while \emph{consecutive} patterns are those of the form $(\sigma, [k-1])$ where the copies of $\sigma$ must appear as subfactors $\pi_i \pi_{i+1} \cdots \pi_{i+k-1} \oi \sigma$.

Classical patterns exhibit several trivial symmetries which extend to vincular patterns as well.  The \emph{reverse} of a permutation $\pi=\pi_1\pi_2 \cdots \pi_n$ is given by $\pi^r=\pi_n\pi_{n-1}\cdots\pi_1$ and the \emph{complement} by $\pi^c=(n+1-\pi_1)(n+1-\pi_2)\cdots(n+1-\pi_n)$.  For vincular pattern $(\sigma, X)$ of length $k$, the reverse of $(\sigma, X)$ is the pattern $(\sigma,X)^{r} = (\sigma^r,  \{k-x : x \in X\})$.  Thus the reverse of $13\da2\da4$ is $4\da2\da31$.  Similarly, the complement of $(\sigma, X)$ is the pattern $(\sigma,X)^{c} = (\sigma^c, X)$.  Thus we see $(13\da2\da4)^c = 42\da3\da1$.  Observe that $\pi$ contains $(\sigma, X)$ if and only if $\pi^r$ contains $(\sigma, X)^r$ and likewise for the complement.

If a permutation $\pi$ does not contain a copy of a vincular pattern $\sigma$, we say that $\pi$ \emph{avoids} $\sigma$.  The subset of $\S{n}$ of permutations avoiding $\sigma$ is denoted $\Sav{n}{\sigma}$.   Two patterns $\sigma$ and $\tau$ are \emph{Wilf-equivalent} if $\bigl|\Sav{n}{\sigma}\bigr| = \bigl|\Sav{n}{\tau}\bigr|$ for all $n\geq 0$, and we denote this by $\sigma \we \tau$.  From the preceding remarks on symmetry it is clear that $\sigma \we \sigma^{r} \we \sigma^{c} \we \sigma^{rc}$.

The results in this paper translate to Wilf-equivalences for barred patterns.   A \emph{barred pattern} is a permutation $\sigma$ where some subset of the letters have bars over them.  Let $\sigma'$ be the classical pattern formed by the unbarred letters of $\sigma$ and let $\sigma''$ be the underlying permutation.  A permutation $\pi$ \emph{avoids} barred pattern $\sigma$ if every copy of $\sigma'$ extends to a copy of $\sigma''$ (including the vacuous case where $\pi$ avoids $\sigma'$ entirely).  For example, consider the barred pattern $\sigma=23\bar{1}4$, so then $\sigma'=123$ and $\sigma''=2314$.  Then we see that $\pi=34152$ avoids $\sigma$, since $\pi_1 \pi_2 \pi_4=345$ is the only copy of $\sigma'$ and it extends to $\pi_1 \pi_2 \pi_3 \pi_4 = 3415 \oi \sigma''$.  On the other hand, $\pi=324651$ does not avoid $\sigma$ since $345 \oi \sigma'$ cannot extend to create a $2314$.  We use the analogous notation $\Sav{n}{\sigma}$ for the permutations avoiding a barred pattern $\sigma$, and define \emph{Wilf-equivalence} for barred patterns in the obvious way.

It is known that there are [sets of] vincular patterns $B$ such that the sets $\Sav{n}{B}$ can also be
characterized by avoiding [sets of] barred patterns.  For example, $\Sav{n}{41\bar{3}52} = \Sav{n}{3\da14\da2}$.  The question of classifying such patterns was raised by Steingr\'{i}msson in \cite{Steingrimsson2010Survey} and
answered by Tenner in \cite{Tenner2013} for avoidance sets classified by avoiding a single vincular pattern.
The vincular patterns in Section \ref{sec:other} each happen to correspond to avoiding a barred pattern, and we state the barred patterns when appropriate.

The paper is divided as follows.  Section \ref{sec:prelims} presents some preliminary work regarding partitioning the set $\Sav{n}{\sigma}$ according to prefix patterns.  These results are used repeatedly in sections \ref{sec:twofamilies} and \ref{sec:umbral}.  Section \ref{sec:twofamilies} proves two infinite families of Wilf-equivalences.  Section \ref{sec:umbral} proves special cases of the families from Section \ref{sec:twofamilies} by a different method which also produces functional equations for the appropriate generating functions.  Section \ref{sec:other} proves several sporadic Wilf-equivalences through a variety of combinatorial methods.  The remaining equivalences to complete the Wilf-classification of vincular patterns of length $4$ are given in the final section.

\section{Partitioning by prefixes}\label{sec:prelims}

Let us make some preliminary observations concerning the partitioning of avoidance classes according to prefixes, which we will make use of in the two subsequent sections.  Given a permutation $p\in \S{k}$, let $\Sp{n}{B}{p}$ be the set of permutations $\pi\in \Sav{n}{B}$ such that $\pi_1 \pi_2 \dotsm \pi_k \oi p$.  For further refinement, let $w\in \{1,2,\dotsc,n\}^k$ and define the set
\begin{equation*}
 \Spt{n}{B}{p}{w} = \{\pi\in \Sp{n}{B}{p}  \colon\, \pi_i = w_i \mathrm{\ for\ } 1\leq i\leq k\}.
\end{equation*}
Obviously $\Spt{n}{B}{p}{w}$ is empty unless $w \oi p$.   For an example, the set of $1\da2\da3$-avoiders of length $5$ starting with ``53'' is:
\begin{equation*}
 \Spt{5}{1\da2\da3}{21}{53} = \{ 53142, 53214, 53241, 53412, 53421 \}.
\end{equation*}
The redundancy of including $p$ in the $\Spt{n}{B}{p}{w}$ notation is maintained to emphasize the subset relation.  In what follows, we will often use $a$, $b$, and $c$ to denote the first, second, and third letters, respectively, of a permutation or prefix.

By looking at the prefix of a permutation, one can identify likely ``trouble spots'' where forbidden patterns may appear.  For example, suppose we wish to avoid the pattern $23\da1$.  Then the presence of the pattern $12$ in the prefix indicates the potential for the entire permutation to contain a $23\da1$ pattern.  In fact, one can quickly determine that $\Spt{n}{23\da1}{12}{ab} = \emptyset$ if $a>1$, since in that case $1$ appears to the right of the prefix and so $ab1$ provides a copy of $23\da1$.  This argument generalizes in the obvious way to the following proposition, which is adapted from Lemma 11 of \cite{Baxter2013}.

\begin{proposition}\label{prop:tailgaps}
 Let $\sigma = \sigma_1 \sigma_2 \dotsm \sigma_k \da \sigma_{k+1}$ be a vincular pattern and let $p\in\S{k}$ such that $p\oi \sigma_1 \sigma_2 \dotsm \sigma_k$.  Then the following conditions each imply that $\Spt{n}{\sigma}{p}{w}= \emptyset$:
\begin{enumerate}
 \item If $\sigma_{k+1}=1$ and $m$ is the index so that $\sigma_{m}= 2$, then $\Spt{n}{\sigma}{p}{w} = \emptyset$ if $w_{m} > 1$.
 \item If $\sigma_{k+1}=k+1$ and $\ell$ is the index so that $\sigma_{\ell} = k$, then $\Spt{n}{\sigma}{p}{w} = \emptyset$ if $w_{\ell} < n$.
 \item If $1 < \sigma_{k+1} < k+1$, $\ell$ is the index so that $\sigma_{\ell} = \sigma_{k+1} -1$, and $m$ is the index so that $\sigma_{m} = \sigma_{k+1} + 1$, then $\Spt{n}{\sigma}{p}{w}=\emptyset$ if $w_{m} - w_{\ell} > 1$.
\end{enumerate}
\end{proposition}

For readers familiar with enumeration schemes, we note that Proposition \ref{prop:tailgaps} is translated from the gap vector criteria for avoiding $\sigma$ of the given form.

We will make use of bijections $\Spt{n}{\sigma}{p}{w} \to \Spt{n'}{\sigma}{p'}{w'}$ which are each done by deleting the first few letters from a permutation.  Specifically, we define the following deletion map:
\begin{definition}
 Let $R\subseteq \{1,2,\dotsc, k\}$ and let $\pi\in\S{n}$ for $n\geq k$.  Define the deletion map $d_R: \S{n} \to \S{n-|R|}$ to act by deleting $\pi_r$ for each $r\in R$, and then reducing the resulting word.
\end{definition}
For example, $d_{\{1,3\}}(3154276)$ deletes the first and third letters, 3 and 5, to form the word $14276$, which is then reduced to the permutation $13254$.  Thus $d_{\{1,3\}}(3154276)=13254$.

It is clear that $d_R : \S{n} \to \S{n-|R|}$ is a surjection such that each image $\pi\in\S{n-|R|}$ has $|R|! \cdot \binom{n}{|R|}$ pre-images.  If one restricts to specific prefixes, however, we can get a bijection.  For a word $w$ with no repeated letters, define $d_R(w)$ to be the word obtained by deleting $w_r$ for each $r\in R$ and subtracting $\bigl|\{r\in R: w_r < w_i\}\bigr|$ from each remaining $w_i$.  For example $d_{\{1,3\}}(63485) = 364$.  Note that this definition is equivalent to the one given above when $w\in\S{n}$.  We now see that restricting the domain of $d_R$ yields a bijection $d_R: \S{n}[p;w] \to \S{n-|R|}[d_R(p);d_R(w)]$.

For certain choices of $p$, $R$, and $\sigma$, applying $d_R$ to a permutation in $\Spt{n}{\sigma}{p}{w}$ results in another $\sigma$-avoiding permutation.  Such an event motivates the following definition:

\begin{definition}
The set of indices $R$ is \emph{reversibly deletable for $p$ with respect to $B$} if the map
\begin{equation}
d_R: \Spt{n}{B}{p}{w} \to \Spt{n-|R|}{B}{d_R(p)}{d_R(w)}
\end{equation}
is a bijection for all $w$ such that $\Spt{n}{B}{p}{w}\neq \emptyset$.
\end{definition}

In the case of patterns of the form $\sigma_1 \sigma_2 \dotsm \sigma_k \da \sigma_{k+1}$, reversibly deletable sets are easy to find.  The following is a strengthening of Lemma 10 in \cite{Baxter2013}.

\begin{proposition}\label{prop:tailrd1}
Let $\sigma = \sigma_1 \sigma_2 \dotsm \sigma_k \da \sigma_{k+1}$ be a vincular pattern and let $p\in \S{t}$ be a prefix pattern for $t\leq k$ such that $\sigma_1 \sigma_2 \dotsm \sigma_t \not\oi p$.  Let $r$ be the smallest index such that $p_{r+1}p_{r+2}\dotsm p_{t} \oi \sigma_1 \sigma_2 \dotsm \sigma_{t-r}$.  Then $R=\{1,2,\dotsc, r\}$ is reversibly deletable for $p$ with respect to $\sigma$.
\end{proposition}

Note that the $r$ in Proposition \ref{prop:tailrd1} is guaranteed to exist since $p_{t} \oi \sigma_1$.  The proof of this proposition follows from the fact that for $\pi\in\Spt{n}{\sigma}{p}{w}$, the letter $\pi_i$ could not be part of a copy of $\sigma$ unless $\pi_i \pi_{i+1} \dotsm \pi_{t} \oi p_i p_{i+1} \dotsm p_{t} \oi \sigma_1 \sigma_2 \dotsm \sigma_{t-i+1}$.  In other words, the constructed deletion map $d_{\{1,2,\dotsc,r\}}$ cuts the permutation just prior to the first letter that a copy of $\sigma$ could possibly use based on the limited information provided by the prefix pattern $p$.

In the case that the entire prefix has length greater than or equal the initial section of $\sigma$, we need the following.

\begin{proposition}\label{prop:tailrd2}
Let $\sigma = \sigma_1 \sigma_2 \dotsm \sigma_k \da \sigma_{k+1}$ be a vincular pattern and let $p=p_1p_2\cdots p_t\in\S{t}$, where $t \geq k$, such that no consecutive $k$ letters of $p$ are isomorphic to $\sigma_1\sigma_2\cdots\sigma_k$ except for possibly the final $k$ letters of $p$.   Let $r\geq t-k+1$ be the smallest index such that $p_{r+1}p_{r+2}\dotsm p_{t} \oi \sigma_1 \sigma_2 \dotsm \sigma_{t-r}$.  Then $R=\{1,2,\dotsc, r\}$ is reversibly deletable for $p$ with respect to $\sigma$.
\end{proposition}

Observe that $\Spt{n}{\sigma}{p}{w}$ is empty in the case that the final $k$ letters of $w$ are isomorphic to $\sigma_1\sigma_2\cdots\sigma_k$ and satisfy the conditions given in Proposition \ref{prop:tailgaps}.  By the definition of reversible deletability, we may focus on the action of $d_R$ for prefixes $w$ which make $\Spt{n}{\sigma}{p}{w}$ non-empty.  We know that if $\pi\in\Spt{n}{\sigma}{p}{w}$, then $\pi_1\pi_2\dotsm \pi_r$ is not part of a copy of $\sigma$, and so we may proceed as in the proof of Proposition \ref{prop:tailrd1}.

\section{Two families of equivalences}\label{sec:twofamilies}

We now consider equivalences involving some general classes of vincular patterns containing a single dash.  Our first result concerns a certain class of vincular patterns having exactly one peak in the initial section.

\begin{theorem}\label{gen2}
Let $k \geq 3$ and $\sigma=\sigma_1\sigma_2\cdots \sigma_k\da\sigma_{k+1}$ denote the vincular pattern of length $k+1$ such that $\sigma_1<\sigma_2<\cdots<\sigma_{i-1}<\sigma_i>\sigma_{i+1}>\cdots>\sigma_k$ and $\sigma_{k+1}=k+1$, where $2 \leq i \leq k-1$.  Let $\tau=\tau_1\tau_2\cdots\tau_k\da\tau_{k+1}$ denote the pattern obtained from $\sigma$ by interchanging the letters $k$ and $k+1$.  Then $\sigma \we \tau$ and this equivalence respects the first letter statistic.
\end{theorem}
\begin{proof}
We will prove $\sigma\equiv\tau$ and that this equivalence respects the first letter statistic together by induction.

We start with the following definition.  By a \emph{block} $B_i$ of a permutation $\pi=\pi_1\pi_2\cdots\pi_n$, we will mean a certain maximal monotonic string of letters defined recursively as follows.  Let $B_1$ be the maximal monotonic initial subfactor of $\pi$.  If $j \geq 2$, then let $B_j$ be the maximal monotonic initial subfactor of the permutation (provided it is non-empty) obtained by removing blocks $B_1, B_2,\ldots, B_{j-1}$ from $\pi$.  For example, if $\pi=148769523 \in \mathcal{S}_9$, then there are four blocks $B_1=148$, $B_2=76$, $B_3=952$ and $B_4=3$.

Let us describe a block as \emph{increasing} or \emph{decreasing} depending on whether or not its letters are in ascending or descending order.  We will often refer to $B_1$ as the initial block.  Given pairwise disjoint subsets $B_1,B_2,\ldots,B_r$ of $[n]$, where the elements of a set may occur in either increasing or decreasing order, and a pattern $\rho$, let $\mathcal{S}_n(\rho)(B_1,B_2,\ldots,B_r)$ denote the subset of $\mathcal{S}_n(\rho)$ whose members contain at least $r$ blocks where the $j$-th block is $B_j$ for $1 \leq j \leq r$.  Given a permutation $\pi=\pi_1\pi_2\cdots\pi_{n}$ and $1 \leq i \leq n-1$, we will say that $\pi$ has a \emph{descent} (\emph{ascent}) at index $i$ if $\pi_i>\pi_{i+1}$ ($\pi_i<\pi_{i+1}$).

Let us assume in the following arguments that an initial block starts with a fixed letter $a$.

Note first that $\mathcal{S}_n(\sigma)(B)$ and $\mathcal{S}_n(\tau)(B)$ are both singletons in the two cases where $B=[n]$ (assuming here either $a=1$ or $a=n$).  Furthermore, since only the final block of a permutation can be a singleton, it follows that $\mathcal{S}_n(\sigma)(B_1,B_2)$ and $\mathcal{S}_n(\tau)(B_1,B_2)$ are the same set when $|B_2|=1$.  So let us assume that permutations contain at least two blocks, with the second block of size at least two.  Clearly, we may also assume $n \geq k+1$ in what follows.

We will now show that $|\mathcal{S}_n(\sigma)(B)|=|\mathcal{S}_n(\tau)(B)|$ for a given decreasing block $B$. Let $b$ denote the last letter of $B$ and $m=n-|B|+1$.  By Proposition \ref{prop:tailrd1} or \ref{prop:tailrd2}, deleting all the letters of $B$ except $b$ results in permutations belonging to $S_m(\sigma)$ and $S_m(\tau)$, which we leave in unreduced form for convenience.
By induction, we know that the total number of members of $S_m(\sigma)[1;b]$  is the same as the total number of members of $S_m(\tau)[1;b]$.  We also have that the permutations in either of these classes which start with a descent are equinumerous, by induction, upon deleting the letter $b$ and considering the resulting permutations of length $m-1$.  By subtraction, the number of members of $S_m(\sigma)[1;b]$ starting with an ascent is the same as the number of members of $S_m(\tau)[1;b]$ starting with an ascent, which implies $|\mathcal{S}_n(\sigma)(B)|=|\mathcal{S}_n(\tau)(B)|$.

If $B$ is increasing with $|B|<i$, then deleting all of the letters of $B$ and considering the first letter of the resulting permutations shows that $|\mathcal{S}_n(\sigma)(B)|=|\mathcal{S}_n(\tau)(B)|$ in this case, by Proposition \ref{prop:tailrd1}.

Now suppose that the initial block $B$ is increasing with $|B|\geq i$ and that the second block is also increasing, where we assume $2 \leq i \leq k-2$ in this case.  By Proposition \ref{prop:tailrd1} or \ref{prop:tailrd2}, we may delete all of the letters of $B$ from the members of both $\mathcal{S}_n(\sigma)(B)$ and $\mathcal{S}_n(\tau)(B)$ to obtain permutations of length $m=n-|B|$, which we leave unreduced.  Note that $m>1$ and suppose that these shorter permutations start with the letter $c$, where $c<\max\{B\}$.  By a subtraction argument similar to the one used above in the case when $B$ was decreasing, the number of members of $S_m(\sigma)[1;c]$ starting with an ascent is the same as the number of members of $S_m(\tau)[1;c]$ starting with an ascent.  Since $c$ is arbitrary, it follows that $|S_n(\sigma)(B)|=|S_n(\tau)(B)|$ in this case.

Next suppose that the initial block $B_1$ is increasing with $|B_1|\geq i$, while the second block $B_2$ is decreasing with $|B_2|< k-i$.  We then delete all the letters of $B_1$ and all but the last letter of $B_2$.  Reasoning as in the case when the initial block $B$ was decreasing shows that $|\mathcal{S}_n(\sigma)(B_1,B_2)|=|\mathcal{S}_n(\tau)(B_1,B_2)|$.

Finally, assume that $B_1$ is increasing with $|B_1|\geq i$ and that $B_2$ is decreasing with $|B_2|\geq k-i$ if $2 \leq i \leq k-2$ (where the same assumptions apply for $B_1$ and $B_2$ if $i=k-1$ except that $B_2$ is allowed to increase as well).  Note that the last letter of $B_1$ in $\pi \in \mathcal{S}_n(\sigma)(B_1,B_2)$ must be $n$ in order to avoid an occurrence of $\sigma$, by (a slight extension of) Proposition \ref{prop:tailgaps}.  Given $B_1$ and $B_2$, let $\widetilde{B}_1$ be obtained by replacing the $n$ at the end of $B_1$ with $\max\{x,y\}+1$, where $x$ is the penultimate letter of $B_1$ and $y$ is the first letter of $B_2$, and let $\widetilde{B}_2=B_2$.  Assume now that $B_2$ is decreasing.  Deleting all letters belonging to the blocks $B_1$ and $B_2$ except for the last letter of $B_2$ from each $\pi \in \mathcal{S}_n(\sigma)(B_1,B_2)$, and doing likewise with $\widetilde{B}_1$ and $\widetilde{B}_2$  for each $\lambda \in \mathcal{S}_n(\tau)(\widetilde{B}_1,\widetilde{B}_2)$, results in members of $\mathcal{S}_m(\sigma)$ and $\mathcal{S}_m(\tau)$ all having the same first letter and starting with an ascent if $m>1$, where $m=n-|B_1|-|B_2|+1$.  Previous arguments now show that the cardinalities of these shorter classes of permutations are the same, which implies $|\mathcal{S}_n(\sigma)(B_1,B_2)|=|\mathcal{S}_n(\tau)(\widetilde{B}_1,\widetilde{B}_2)|$.  If $i=k-1$ and $B_2$ is increasing, then a similar argument applies and yields the same conclusion, this time upon deleting all letters in $B_1$ and $B_2$.  Note further that $\widetilde{B}_1$ and $\widetilde{B}_2$ range over all possibilities for the first two blocks of $\lambda \in \mathcal{S}_n(\tau)(S,T)$, where $|S| \geq i$ and $|T|\geq k-i$ with $S$ increasing and $T$ decreasing if $2 \leq i \leq k-2$ (with $T$ allowed also to increase if $i=k-1$), as $B_1$ and $B_2$ range over all possibilities for the first two blocks of $\pi \in \mathcal{S}_n(\sigma)$ satisfying the same requirements.

Collecting all of the cases above in which a permutation starts with the given letter $a$ completes the induction and the proof.
\end{proof}

\emph{Remark:}  The preceding proof shows further that the equivalence respects descent sets (and hence the number of descents).

As special cases of Theorem \ref{gen2}, we get equivalences such as the following:
\begin{align*}
&\bullet1243\da5\we1253\da4\\
&\bullet1432\da5\we1532\da4\\
&\bullet2431\da5\we2531\da4.
\end{align*}

We recall the following general result.

\begin{theorem}\label{KE}
(Kitaev \cite{Kitaev2005}, Elizalde \cite{Elizalde2006}.) If $\alpha$ and $\beta$ are consecutive patterns of length $k$ and $\alpha \we \beta$, then $\alpha\da(k+1) \we \beta\da(k+1)$.
\end{theorem}

As a corollary of Theorems \ref{gen2} and \ref{KE}, we obtain the following result.

\begin{corollary}\label{cor1}
If $\tau$ has the form described in Theorem \ref{gen2}, then $\tau_1\tau_2\cdots\tau_k\da k \we \tau_k\cdots \tau_2\tau_1\da k$.
\end{corollary}

Our next result concerns a class of patterns whose section to the left of the dash is increasing and contains all but one letter.

\begin{theorem}\label{gen1}
If $k \geq 3$ and $1 \leq i \leq k-1$, then let $\sigma_i$ denote the vincular pattern of length $k+1$ given by $\sigma_i=12\cdots i(i+2)\cdots(k+1)\da(i+1)$.  Then $\sigma_i\equiv\sigma_j$ for all $1 \leq i,j \leq k-1$.  Furthermore, these equivalences respect the first letter statistic.
\end{theorem}
\begin{proof}

Given $k \geq 3$ and $1 \leq i \leq k-1$, let $\sigma_i$ denote the vincular pattern $12\cdots i(i+2)(i+3)\cdots (k+1)$-$(i+1)$.  We will show by induction that the number of permutations of length $n$ which avoid $\sigma_i$ and start with a given letter is the same for all $i$.  If $n<k+1$, then the result is clear, so let us assume $n\geq k+1$.

First note that the set of permutations of length $n$ avoiding $\sigma_i$ and starting $x_1<x_2<\cdots<x_j>x_{j+1}$, where $x_1,x_2,\ldots,x_{j+1}$ are given and $1 \leq j <k$, have the same cardinality for all $i$ by induction.  This can be seen upon deleting the first $j$ letters, which by Proposition \ref{prop:tailrd1} are extraneous concerning the avoidance or containment of any of the $\sigma_i$, and considering the permutation of $[n-j]$ that results whose first element is fixed (i.e., the element corresponding to  $x_{j+1}$ in the reduction of the set $[n]-\{x_1,x_2,\ldots,x_j\}$).

So assume that the permutations $\pi=\pi_1\pi_2\cdots\pi_n$ under consideration which avoid $\sigma_i$ contain at least one descent and have their first descent at index $\ell$, where $k \leq \ell \leq n-1$ is given.  Let us refine this set further by assuming $\pi_1=a$ and $\pi_{\ell}=b$.
Note that in order for $\pi=\pi_1\pi_2\cdots\pi_n$ to avoid $\sigma_i$, have its first descent at index $\ell$, and satisfy $\pi_1=a$ and $\pi_\ell=b$, the following conditions must hold concerning the first $\ell$ letters: (i) $a=\pi_1<\pi_2<\cdots<\pi_i$, (ii) $\pi_{i+r}=\pi_i+r$ for $1 \leq r \leq \ell-k$, and (iii) $\pi_{\ell-(k-i-1)}<\pi_{\ell-(k-i-2)}<\cdots<\pi_\ell=b$.

We will show that the number of permutations $\pi=\pi_1\pi_2\cdots\pi_n$, which avoid $\sigma_i$ and have their first descent at index $\ell$ with $\pi_1=a$ and $\pi_\ell=b$, does not depend on $i$ and is hence the same for all $i$.  To do so, it suffices to show (which is done in the next-to-last paragraph) that the number of possible subsets $S$ comprising $\pi_1\pi_2\cdots\pi_{\ell}$ and satisfying conditions (1), (2), and (3) does not depend on $i$.

To see why, note first that if the initial $\ell$ letters of a permutation $\pi$ satisfy the three conditions above, with $\pi_{\ell+1}<\pi_\ell$, then each of the first $\ell$ letters are seen to be extraneous concerning the avoidance of $\sigma_i$ and thus may be deleted.  Regardless of $i$ and the choice of letters for $S$, there are $\pi_\ell-\ell=b-\ell$ choices for $\pi_{\ell+1}$ in the permutation $\pi_{\ell+1}\pi_{\ell+2}\cdots\pi_n$.  Note further that if $\pi_{\ell+1}\pi_{\ell+2}\cdots\pi_n$ is reduced to a member of $\mathcal{S}_{n-\ell}$, then $\pi_{\ell+1}$ would correspond to a member of $[b-\ell]$.  Suppose now $\pi=\pi_1\pi_2\cdots\pi_n$ and $\rho=\rho_1\rho_2\cdots\rho_n$ are permutations avoiding $\sigma_{i_1}$ and $\sigma_{i_2}$, where $1\leq i_1 <i_2\leq k-1$, with $S$ and $T$ the subsets of $[n]$ corresponding to the first $\ell$ letters and satisfying conditions (1)-(3) above but with $i=i_1$ and $i=i_2$, respectively.  Then there are $b-\ell$ choices for the letter $\pi_{\ell+1}$ in $\pi$ and the same number for $\rho_{\ell+1}$ in $\rho$.  Furthermore, if $\pi_{\ell+1}$ corresponds to the same letter in $\text{red}(\pi_{\ell+1}\pi_{\ell+2}\cdots\pi_n)$ as does $\rho_{\ell+1}$ in
$\text{red}(\rho_{\ell+1}\rho_{\ell+2}\cdots\rho_n)$, we see that the number of possibilities concerning the final $n-\ell$ letters of $\pi$ or $\rho$ is the same in this case, by induction. Since $\pi_{\ell+1}$ and $\rho_{\ell+1}$ always correspond to a member of $[b-\ell]\subseteq[n-\ell]$, it follows that the number of possibilities for $\pi_{\ell+1}\pi_{\ell+2}\cdots\pi_n$, given $S$, is the same as the number of possibilities for $\rho_{\ell+1}\rho_{\ell+2}\cdots\rho_n$, given $T$ (in fact, this number is the same for all $S$ and $T$).  Since the number of possible $S$ and $T$ is always the same for $i_1$ and $i_2$ as the parameters $b$ and $\ell$ vary, it follows that number of permutations of length $n$ starting with $a$ and having an initial increasing subsequence of length at least $k$ and avoiding either $\sigma_{i_1}$ or $\sigma_{i_2}$ is the same.

Let us now count the number of subsets $S$ satisfying conditions (1)-(3) above, where $1 \leq i \leq k-1$.  If $i=1$, then $\pi_1=a$ implies $\pi_{r+1}=a+r$, $1 \leq r \leq \ell-k$, so that $\pi_\ell=b$ implies there are $b-1-(a+\ell-k)$ letters to choose from for the $k-2$ letters directly preceding the last.  Thus, there are $\binom{b+k-a-\ell-1}{k-2}$ possible $S$ when $i=1$.  Now assume $2 \leq i \leq k-1$.  Let $\pi_i=t$ in $\pi_1\pi_2\cdots\pi_\ell$.  Then there are $\binom{t-a-1}{i-2}$ choices for the section $\pi_2\pi_3\cdots \pi_{i-1}$ since $\pi_1=a$ and $\pi_i=t$ and $\binom{b-(\ell-k)-t-1}{k-i-1}$ choices for the section $\pi_{\ell-(k-i-1)}\pi_{\ell-(k-i-2)}\cdots\pi_{\ell-1}$ since all members of the set $\{t+1,t+2,\ldots,t+\ell-k\}$ must be selected.  Note that $t$ can range from $i+a-1$ to $i+b-\ell$. Thus, there are
$$\sum_{t=i+a-1}^{i+b-\ell}\binom{t-a-1}{i-2}\binom{b+k-\ell-t-1}{k-i-1}=\binom{b+k-a-\ell-1}{k-2}$$
possibilities for $S$ in all, where we have used identity (5.26) in \cite[p.169]{Graham1994}.  Therefore, the number of subsets satisfying conditions (1)-(3) does not depend on $i$, as desired.

Collecting the cases above in which a permutation starts with a given letter completes the induction and the proof.
\end{proof}

\emph{Remark:}  Upon fixing $\ell$ and allowing $a$ and $b$ to vary in the preceding proof, one can show by induction that the equivalence respects descents sets.

As special cases of Theorem \ref{gen1}, we get
\begin{align*}
&\bullet 1235\da4 \we 1245\da3 \we 1345\da2\\
&\bullet 12346\da5 \we 12356\da4 \we 12456\da3 \we 13456\da2.\\
\end{align*}

\section{Equivalences via umbral operators}\label{sec:umbral}

In this section we will prove three equivalences, which are each special cases of either Theorem \ref{gen2} or \ref{gen1}:
\begin{subequations}\label{eqn:umbralequivs}
 \begin{align}
 134\da2 &\equiv 124\da3 \\
 132\da4 &\equiv 142\da3 \\
 231\da4 &\equiv 241\da3.
\end{align}
\end{subequations}
We provide alternate proofs by showing that the generating function for the number of permutations avoiding $\sigma$ satisfies the same system of functional equations as the generating function for the number of permutations avoiding $\tau$.  The systems of functional equations are adapted from enumeration schemes as developed in \cite{Baxter2013} and translated into the language of transfer matrices as per Zeilberger's Umbral Transfer Matrix Method developed in \cite{Zeilberger2000}.

\subsection{124-3 \bf{$\equiv$} 134-2}

In this subsection, we will prove Theorem \ref{thmfeq1}, below, which has as a corollary a special case of Theorem \ref{gen1}.

For either $\sigma \in \{124\da3, 134\da2\}$, Proposition \ref{prop:tailrd1} implies that the following sets are reversibly deletable for the given prefixes with respect to $\sigma$:
\begin{enumerate}
 \item $R=\{1\}$ for $p=21$
 \item $R=\{1,2\}$ for $p=132$
 \item $R=\{1,2\}$ for $p=231$.
\end{enumerate}

These imply the following maps are bijections for either $\sigma \in \{124\da3, 134\da2\}$:
\begin{enumerate}
 \item $d_1 : \Spt{n}{\sigma}{21}{ab} \to \Spt{n-1}{\sigma}{1}{b}$
 \item $d_{\{1,2\}} : \Spt{n}{\sigma}{132}{abc} \to \Spt{n-2}{\sigma}{1}{c-1}$
 \item $d_{\{1,2\}} : \Spt{n}{\sigma}{231}{abc} \to \Spt{n-2}{\sigma}{1}{c}$.
\end{enumerate}

Proposition \ref{prop:tailrd2} tells us that $R=\{1\}$ is reversibly deletable for $p=123$ for either $\sigma \in \{124\da3, 134\da2\}$.  It can be proven\footnote{For example, the \texttt{isRevDelSetS} procedure from the Maple package \texttt{gVatter} provides rigorous proof by checking finitely many cases.}, however, that $R=\{2\}$ is also reversibly deletable for $p=123$ in this case.  Combining the bijectivity of $d_{\{2\}}$ with Proposition \ref{prop:tailgaps} tells us the following:

\begin{equation}
 \bigl| \Spt{n}{124\da3}{123}{abc}  \bigr| =
\begin{cases}
 0 & c-b > 1 \\
\bigl|  \Spt{n-1}{124\da3}{12}{a(c-1)}  \bigr| & c-b=1
\end{cases}
\end{equation}
\begin{equation}
 \bigl| \Spt{n}{134\da2}{123}{abc}  \bigr| =
\begin{cases}
 0 & b-a > 1 \\
\bigl|  \Spt{n-1}{134\da2}{12}{a(c-1)}  \bigr| & b-a=1,
\end{cases}
\end{equation}
where in each case the correspondence is performed via the $d_{\{2\}}$ deletion map.

We now invert these bijective deletion maps to build up larger permutations from smaller.  For example, the deletion  $d_1 : \Spt{n}{\sigma}{21}{ab} \to \Spt{n-1}{\sigma}{1}{b}$ is a bijection for a fixed $a$, and so we can invert it and re-index to get the maps
$d_{1}^{-1}: \Spt{n}{\sigma}{1}{b} \to \Spt{n+1}{\sigma}{21}{ab}$ for any $a\in \{b+1, b+2, \dotsc, n+1\}$.  We may likewise convert each of the other deletion maps above into a collection of insertions.

Let $S_1 := \S{1} \cup \bigcup_{n\geq 2} \Sp{n}{\sigma}{21}$ and $S_2 := \bigcup_{n\geq 2} \Sp{n}{\sigma}{12}$, so that $\bigcup_{n\geq 1} \Sav{n}{\sigma}$ is partitioned into $S_1 \cup S_2$.    For a permutation in $S_1$ or $S_2$, we may insert letters at the front according to the inverses of the deletion maps discussed above.  These have the following consequences:
\begin{enumerate}
 \item  Insert a single letter at the front of $\pi\in S_1$ to create a permutation with prefix pattern $21$, which lies in $S_1$.
 \item  Insert two letters at the front of $\pi\in S_1$ to make a permutation pattern with prefix pattern $132$, which lies in $S_2$.
 \item  Insert two letters at the front of $\pi\in S_1$ to make a permutation pattern with prefix pattern $231$, which lies in $S_2$.
 \item  Insert a letter in the second position of $\pi\in S_2$ to make a permutation pattern with prefix pattern $123$, which lies in $S_2$.
\end{enumerate}

It is easily seen that every permutation in $S_1$ other than $1$ itself starts with a $21$ pattern, in which case it is an image of an insertion of the first kind above.  Likewise, every permutation in $S_2$  other than $12$ starts with the prefix pattern $132$, $231$, or $123$ and thus is an image of one of the insertions of the second, third, or fourth kind above.

We now convert these insertions into operators on weight-enumerators, as per the Umbral Transfer Matrix Method of Zeilberger \cite{Zeilberger2000} and applied to enumeration schemes in \cite{BaxterThesis}.  Define the \emph{weight} of a permutation $\pi=\pi_1\dotsm\pi_n$ by $W(\pi) = z^n\,x^{\pi_1}$, and the weight of a set of permutations to be the sum of the weights of its elements.  We then define the following generating functions:
 \begin{equation}
 \begin{split}
 \F^{\sigma}_{1}(z,x) &:= W(S_1) =  z^{1}\,x^{1} + \sum_{n\geq 2} \sum_{\pi\in \Sp{n}{\sigma}{21}} z^n x^{\pi_1} \\
 \F^{\sigma}_{2}(z,x)  &:= W(S_2) = \sum_{n\geq 2} \sum_{\pi\in \Sp{n}{\sigma}{12}} z^n x^{\pi_1} \\
 \F^{\sigma}(z,x) &:= W(S_1 \cup S_2 ) =  \sum_{n\geq 1} \sum_{\pi\in \Sav{n}{\sigma}} z^n x^{\pi_1}. \\
\end{split}
\end{equation}
Note that $\F^{\sigma}(z,x) =  \F^{\sigma}_{1}(z,x) +  \F^{\sigma}_{2}(z,x)$.

The four insertion maps translate to functional equations as follows.  The map which inserts a letter $a$ at the front of $\pi\in S_1$ to create a permutation with prefix pattern $21$ maps a permutation of weight $z^n x^b$ to a permutation of weight $z^{n+1}\,x^a$.  Summing over all possible values of $a$, we see that the weight of the set of  images of a permutation with weight $z^n x^b$ is given by
\begin{equation}
 \begin{split}
   \sum_{a=b+1}^{n+1} z^{n+1}\,x^{a} &= \frac{z^{n+1}\,x^{n+2}}{x-1} - \frac{z^{n+1}\,x^{b+1}}{x-1}\\
   &=  \frac{zx^2}{x-1}\,(zx)^{n} - \frac{zx}{x-1}\, z^n\,x^b.\\
\end{split}
\end{equation}
Define the linear operator $\Pop_{21}$ on monomials in $\mathbb{Q}[[z,x]]$ by
$$\Pop_{21}: z^n\,x^b \mapsto \frac{zx^2}{x-1}\,(zx)^{n} - \frac{zx}{x-1}\, z^n\,x^b.$$
We then extend $\Pop_{21}$ linearly to any bivariate power series $G(z,x)\in \mathbb{Q}[[z,x]]$ over its basis of monomials $z^n x^b$ to see the action:
\begin{equation}
 \begin{split}
  \Pop_{21}\circ G(z,x) &=  \frac{zx^2}{x-1} G(zx,1) - \frac{zx}{x-1} G(z,x) \\
                                     &= \frac{zx}{x-1} \Bigl( x\,G(zx,1) - G(z,x) \Bigr). \\
 \end{split}
\end{equation}

Similarly, the insertions for the prefix patterns $132$ and $231$ can be translated into operators on monomials
\begin{equation}
 \begin{split}
  \Pop_{132}&: z^n\,x^c \mapsto  \sum_{a=1}^{c} \sum_{b=c+2}^{n+2}  z^{n+2} x^a \\
  \Pop_{231}&: z^n\,x^c \mapsto  \sum_{a=c+1}^{n+1} \sum_{b=a+1}^{n+2} z^{n+2} x^a,\\
 \end{split}
\end{equation}
which extend linearly to power series $G(z,x)$:
\begin{equation}
 \begin{split}
  \Pop_{132}\circ G(z,x) &= \frac{z^2\,x}{x-1} \Bigl(  G(z,x) - G(z,1)+ G_x(z,1) - G_x(z,x) + G_z(z,x)  - G_z(z,1)\Bigr)\\
  \Pop_{231}\circ G(z,x) &= \frac{z^2\,x}{(x-1)^2} \Bigl( (x-1)^2\,G(z,x)+ (1-x)\,G_z(z,x) + (x-1)\,G_x(z,x)  \Bigr).\\
 \end{split}
\end{equation}

Lastly, the insertions for the prefix pattern $123$, whether we are avoiding $124\da3$ or $134\da2$, yield the operator
$$ \Pop_{123}: z^n\,x^a \mapsto z^{n+1}\,x^a,$$
since in each case we may insert the $b$ into $\pi\in\Spt{n}{\sigma}{12}{a(c-1)}$ in only one way: if $\sigma=134\da2$ then  $b=a+1$, or if $\sigma=124\da3$ then $b=c-1$.  This extends linearly to power series $G(z,x)$ by
$$ \Pop_{123}\circ G(z,x) = z\,G(z,x).$$

Since the operators $\Pop_{21}$, $\Pop_{132}$, $\Pop_{231}$, and $\Pop_{123}$ track the effects of these insertions on the weight of permutations, we obtain the following system of functional equations.

\begin{theorem}\label{thmfeq1}
We have
\begin{equation}
\left\{
\begin{array}{rl}
 \F^{\sigma}_{1}(z,x) &= z\,x+ \Pop_{21} \circ \F^{\sigma}_{1}(z,x) \\
 \F^{\sigma}_{2}(z,x) &= z^2\,x +\Pop_{132} \circ \F^{\sigma}_{1}(z,x) + \Pop_{231} \circ \F^{\sigma}_{1}(z,x)  + \Pop_{123} \circ \F^{\sigma}_{2}(z,x), \\
\end{array}
\right.
\end{equation}
where the $\Pop_\tau$ operators are as defined above.
\end{theorem}
Since these operators are the same regardless of whether $\sigma = 124\da3$ or $\sigma = 134\da2$, we have
$$
\F^{124\da3}(z,x) = \F^{134\da2}(z,x),
$$ which implies the following result.

\begin{corollary}\label{thm:equiv1}
 $124\da3 \we 134\da2$.  Furthermore, this equivalence respects the first letter statistic.
\end{corollary}

It should be noted that the above arguments can yield functional equations for $\F^{\sigma}(z,x)$ of the form described in Theorem \ref{gen1} and can provide a generating function proof of that statement.  However, the sets $\sum_{n\geq 1} \Sav{n}{\sigma}$ must be partitioned further into $S_1, S_2, \dotsc, S_k$ according to the location of their first descent, where $S_i$ also contains the monotone permutation $12\dotsm i$ and $S_k$ contains all permutations with no descents before the $k$-th position.

\subsection{132-4 \bf{$\equiv$} 142-3}

In this subsection, we employ the same methods as in the previous to prove a special case of Theorem \ref{gen2} via generating functions.

For $\sigma \in \{132\da4, 142\da3\}$, it can be seen from Propositions \ref{prop:tailrd1} and \ref{prop:tailrd2} that the following sets are reversibly deletable for the given prefixes with respect to $\sigma$:
\begin{enumerate}
 \item $R=\{1\}$ for $p=21$
 \item $R=\{1\}$ for $p=123$
 \item $R=\{1,2\}$ for $p=231$
 \item $R=\{1,2\}$ for $p=132$.
\end{enumerate}

These imply the following maps are bijections for either $\sigma \in \{132\da4, 142\da3\}$:
\begin{enumerate}
 \item $d_1 : \Spt{n}{\sigma}{21}{ab} \to \Spt{n-1}{\sigma}{1}{b}$
 \item $d_1 : \Spt{n}{\sigma}{123}{abc} \to \Spt{n-1}{\sigma}{12}{(b-1)(c-1)}$
 \item $d_{\{1,2\}} : \Spt{n}{\sigma}{231}{abc} \to \Spt{n-2}{\sigma}{1}{c}$.
\end{enumerate}

Combining the above with Proposition \ref{prop:tailgaps}, we also get the following equalities, where in each case the correspondence is performed via $d_{\{1,2\}}$:
\begin{equation}
 \bigl| \Spt{n}{132\da4}{132}{abc}  \bigr| =
\begin{cases}
 0 &  b<n \\
\bigl|  \Spt{n-2}{132\da4}{1}{c-1}  \bigr| & b=n
\end{cases}
\end{equation}
\begin{equation}
 \bigl| \Spt{n}{142\da3}{132}{abc}  \bigr| =
\begin{cases}
 0 & b-c > 1 \\
\bigl|  \Spt{n-2}{142\da3}{1}{c-1}  \bigr| & b-c=1.
\end{cases}
\end{equation}

As before, let $S_1 = \S{1} \cup \bigcup_{n\geq 2} \Sp{n}{\sigma}{21}$ and $S_2 = \bigcup_{n\geq 2} \Sp{n}{\sigma}{12}$.  For a permutation in $S_1$ or $S_2$, we may insert letters at the front according to the inverses of the deletion maps discussed above.  These have the following consequences:
\begin{enumerate}
 \item  Insert a single letter at the front of $\pi\in S_1$ to create a permutation with prefix pattern $21$, which lies in $S_1$.
 \item  Insert two letters at the front of $\pi\in S_1$ to make a permutation pattern with prefix pattern $132$, which lies in $S_2$.
 \item  Insert two letters at the front of $\pi\in S_1$ to make a permutation pattern with prefix pattern $231$, which lies in $S_2$.
 \item  Insert a letter at the front of $\pi\in S_2$ to make a permutation pattern with prefix pattern $123$, which lies in $S_2$.
\end{enumerate}

We now convert these insertions into operators on weight-enumerators $\F^{\sigma}_{1}(z,x)$,  $\F^{\sigma}_{2}(z,x)$, and $\F^{\sigma}(z,x)$, as in the previous section, where again the \emph{weight} of a permutation $\pi_1\dotsm\pi_n$ is given by $W(\pi) = z^n\,x^{\pi_1}$.

As before, the insertions for the prefix patterns $21$, $123$, and $231$ can be translated into linear operators on power series by first defining their action on monomials:
\begin{equation}
 \begin{split}
  \Pop_{21}&: z^n\,x^b \mapsto  \sum_{a=b+1}^{n+1} z^{n+1} x^a \\
  \Pop_{123}&: z^n\,x^b \mapsto  \sum_{a=1}^{b} z^{n+1} x^a \\
  \Pop_{231}&: z^n\,x^c \mapsto  \sum_{a=c+1}^{n+1} \sum_{b=a+1}^{n+2} z^{n+2} x^a,\\
 \end{split}
\end{equation}
and then extending to power series $G(z,x)\in\mathbb{Q}[[z,x]]$:
\begin{equation}
 \begin{split}
  \Pop_{21}\circ G(z,x) &= \frac{zx}{x-1} \Bigl( x\,G(zx,1) - G(z,x) \Bigr) \\
  \Pop_{123}\circ G(z,x) &= \frac{z\,x}{x-1} \Bigl(  G(z,x)  - G(z,1) \Bigr)\\
  \Pop_{231}\circ G(z,x) &= \frac{z^2\,x}{(x-1)^2} \Bigl(  (1-2x) G(z,x) + x^2\,G(zx,1) + (1-x)G_z(z,x) \\
  &\quad + (x-1)\,G_x(z,x)    \Bigr).\\
 \end{split}
\end{equation}

Concerning insertions for the prefix pattern $132$, it appears \emph{a priori} that the operator $\Pop_{132}$ depends on whether we are avoiding $132\da4$ or $142\da3$.  Let $\Pop'_{132}$ be the operator for when we avoid $132\da4$ and let $\Pop''_{132}$ be the operator for when we avoid $142\da3$.  Then by the same reasoning as above, we see that these operators are as follows:
$$ \Pop'_{132}: z^n\,x^c \mapsto \sum_{a=1}^{c} \sum_{b=n+2}^{n+2} z^{n+2}\,x^a $$
$$ \Pop''_{132}: z^n\,x^c \mapsto \sum_{a=1}^{c} \sum_{b=c+2}^{c+2} z^{n+2}\,x^a.$$
In each case, however, the inside summation works out the same, and so $\Pop'_{132} = \Pop''_{132}$.  We merge these symbols into the single operator on monomials
$$ \Pop_{132}: z^n\,x^c \mapsto \sum_{a=1}^{c} z^{n+2}\,x^a,$$
which extends linearly to power series:
$$ \Pop_{132}\circ G(z,x) = \frac{z^2\,x}{x-1} \Bigl(  G(z,x)  - G(z,1) \Bigr).$$

Since the operators $\Pop_{21}$, $\Pop_{132}$, $\Pop_{231}$, and $\Pop_{123}$ track the effects of insertions on the weight of permutations when $\sigma=132\da4$ or $\sigma = 142\da3$, we obtain the following system of functional equations.

\begin{theorem}\label{thmfeq2}
We have
\begin{equation}\label{eqn:PopSystem}
\left\{
\begin{array}{rl}
 \F^{\sigma}_{1}(z,x) &= z\,x + \Pop_{21} \circ \F^{\sigma}_{1}(z,x) \\
 \F^{\sigma}_{2}(z,x) &= z^2\,x +\Pop_{132} \circ \F^{\sigma}_{1}(z,x) + \Pop_{231} \circ \F^{\sigma}_{1}(z,x)  + \Pop_{123} \circ \F^{\sigma}_{2}(z,x),\\
\end{array}
\right.
\end{equation}
where the $\Pop_\tau$ operators are as defined above.
\end{theorem}

Since the operators $\Pop_{\tau}$ are the same regardless of whether $\sigma = 132\da4$ or $\sigma = 142\da3$, we have shown
$$
\F^{132\da4}(z,x) = \F^{142\da3}(z,x),
$$
which implies the following result.

\begin{corollary}\label{thm:equiv2}
 $132\da4 \we 142\da3$.  Furthermore, this equivalence respects the first letter statistic.
\end{corollary}

\subsection{231-4 \bf{$\equiv$} 241-3}

By methods similar to those above, we prove another special case of Theorem \ref{gen2}.

Propositions \ref{prop:tailrd1} and \ref{prop:tailrd2} imply that the following sets are reversibly deletable for the given prefixes with respect to $\sigma\in \{231\da4, 241\da3\}$:
\begin{enumerate}
 \item $R=\{1\}$ for $p=21$
 \item $R=\{1\}$ for $p=123$
 \item $R=\{1,2\}$ for $p=132$
 \item $R=\{1,2\}$ for $p=231$.
\end{enumerate}

These imply the following maps are bijections for $\sigma = 241\da3$ and for $\sigma = 231\da4$:
\begin{enumerate}
 \item $d_1 : \Spt{n}{\sigma}{21}{ab} \to \Spt{n-1}{\sigma}{1}{b}$
 \item $d_1 : \Spt{n}{\sigma}{123}{abc} \to \Spt{n-1}{\sigma}{12}{(b-1)(c-1)}$
 \item $d_{\{1,2\}} : \Spt{n}{\sigma}{132}{abc} \to \Spt{n-2}{\sigma}{1}{c-1}$.
\end{enumerate}

Combining the above with Proposition \ref{prop:tailgaps}, we also get the following equalities, where in each case the correspondence is performed via $d_{\{1,2\}}$:
\begin{equation}
 \bigl| \Spt{n}{231\da4}{231}{abc}  \bigr| =
\begin{cases}
 0 &  b<n \\
\bigl|  \Spt{n-2}{231\da4}{1}{c-1}  \bigr| & b=n
\end{cases}
\end{equation}
\begin{equation}
 \bigl| \Spt{n}{241\da3}{231}{abc}  \bigr| =
\begin{cases}
 0 & b-a > 1 \\
\bigl|  \Spt{n-2}{241\da3}{1}{c-1}  \bigr| & b-a=1.
\end{cases}
\end{equation}

As before, let $S_1 = \S{1} \cup \bigcup_{n\geq 2} \Sp{n}{\sigma}{21}$ and $S_2 = \bigcup_{n\geq 2} \Sp{n}{\sigma}{12}$.  For a permutation in $S_1$ or $S_2$, we may insert letters at the front according to the inverses of the deletion maps discussed above.  Therefore we have the following insertions:
\begin{enumerate}
 \item  Insert a single letter at the front of $\pi\in S_1$ to create a permutation with prefix pattern $21$, which lies in $S_1$.
 \item  Insert two letters at the front of $\pi\in S_1$ to make a permutation pattern with prefix pattern $132$, which lies in $S_2$.
 \item  Insert two letters at the front of $\pi\in S_1$ to make a permutation pattern with prefix pattern $231$, which lies in $S_2$.
 \item  Insert a letter at the front of $\pi\in S_2$ to make a permutation pattern with prefix pattern $123$, which lies in $S_2$.
\end{enumerate}

Defining the generating functions $\F^{\sigma}_{1}(z,x)$, $\F^{\sigma}_{2}(z,x)$, and $\F^{\sigma}(z,x)$ as before, we now translate the above insertions into linear operators on power series in $\mathbb{Q}[[z,x]]$.  In this case the actions on monomials are given by:
\begin{equation}
 \begin{split}
  \Pop_{21}&: z^n\,x^b \mapsto  \sum_{a=b+1}^{n+1} z^{n+1} x^a \\
  \Pop_{123}&: z^n\,x^b \mapsto  \sum_{a=1}^{b} z^{n+1} x^a \\
  \Pop_{132}&: z^n\,x^c \mapsto  \sum_{a=1}^{c} \sum_{b=c+2}^{n+2} z^{n+2} x^a.\\
 \end{split}
\end{equation}

Analogous to the arguments in the previous section, the insertions for the prefix pattern $231$ may yield different operators $\Pop_{231}$ depending on whether we are avoiding $231\da4$ or $241\da3$.  Let $\Pop'_{231}$ be the operator for when we avoid $231\da4$ and let $\Pop''_{231}$ be the operator for when we avoid $241\da3$.  Then by the same reasoning as above, we obtain the formulas:
$$ \Pop'_{231}: z^n\,x^c \mapsto \sum_{a=c+1}^{n+1} \sum_{b=n+2}^{n+2} z^{n+2}\,x^a $$
$$ \Pop''_{231}: z^n\,x^c \mapsto \sum_{a=c+1}^{n+1} \sum_{b=c+2}^{c+2} z^{n+2}\,x^a. $$
In each case, however, the inside summation works out the same, and so $\Pop'_{231} = \Pop''_{231}$.  We merge these operators into the single operator
$$ \Pop_{231}: z^n\,x^c \mapsto  \sum_{a=c+1}^{n+1}  z^{n+2} x^a.$$

The $\Pop_{p}$ operators above extend linearly to power series as follows:
\begin{equation}
 \begin{split}
  \Pop_{21}\circ G(z,x) &=  \frac{zx}{x-1} \Bigl(x\, G(zx,1) -  G(z,x) \Bigr) \\
  \Pop_{123}\circ G(z,x) &= \frac{z\,x}{x-1} \Bigl(  G(z,x)  - G(z,1) \Bigr)\\
  \Pop_{132}\circ G(z,x) &= \frac{z^2\,x}{x-1} \Bigl( G(z,x) + G_x(z,1) + G_z(z,x) - G(z,1) - G_x(z,x) - G_z(z,1) \Bigr)\\
  \Pop_{231}\circ G(z,x) &= \frac{z^2\,x}{x-1} \Bigl( x\,G(zx,1) - G(z,x)  \Bigr).\\
 \end{split}
\end{equation}

This implies the following system of equations for $\sigma \in \{231\da4, 241\da3\}$.

\begin{theorem}\label{thmfeq3}
We have
\begin{equation}\label{eqn:PopSystem}
\left\{
\begin{array}{rl}
 \F^{\sigma}_{1}(z,x) &= z\,x +  \Pop_{21} \circ \F^{\sigma}_{1}(z,x) \\
 \F^{\sigma}_{2}(z,x) &= z^2\,x + \Pop_{132} \circ \F^{\sigma}_{1}(z,x) + \Pop_{231} \circ \F^{\sigma}_{1}(z,x)  + \Pop_{123} \circ \F^{\sigma}_{2}(z,x), \\
\end{array}
\right.
\end{equation}
where the $\Pop_\tau$ operators are as defined above.
\end{theorem}

Since the operators $\Pop_{\tau}$ are independent of $\sigma = 231\da4$ or $\sigma = 241\da3$, we have shown
$$
\F^{231\da4}(z,x) = \F^{241\da3}(z,x),
$$
which implies the following result.

\begin{corollary}\label{thm:equiv3}
 $231\da4 \we 241\da3$.  Furthermore, this equivalence respects the first letter statistic.
\end{corollary}


\section{Other equivalences}\label{sec:other}

In this section, we consider some additional equivalences involving vincular patterns of length 4.  When taken with the results of the prior section and of previous papers, one almost completes the Wilf-classification of vincular patterns of length 4.  Throughout this section, we let $[m,n]=\{m,m+1,\ldots,n\}$ if $m \leq n$ are positive integers, with $[m,n]=\emptyset$ if $m>n$.

\subsection{1-24-3 \bf{$\equiv$} 1-42-3} 

The following result answers in the affirmative Conjecture 17(d) which was raised in \cite{Baxter2012}.

\begin{theorem}\label{1-24-3}
$1\da24\da3\equiv 1\da42\da3$.  Furthermore, this equivalence respects both the positions and values of left-to-right minima.
\end{theorem}
\begin{proof}
We define a bijection $f=f_n$ between the sets $\mathcal{S}_n(1\mbox{-}42\mbox{-}3)$ and $\mathcal{S}_n(1\mbox{-}24\mbox{-}3)$.
To do so, let us first decompose $\pi \in \mathcal{S}_n(1\mbox{-}42\mbox{-}3)$ as
$$\pi=m_t\alpha_tm_{t-1}\alpha_{t-1}\cdots m_1\alpha_1,$$
where $t \geq 1$ and $m_t>m_{t-1}>\cdots>m_1$ are the left-to-right minima of $\pi$.  Note that $\alpha_1$ must avoid the pattern 31-2 since $m_1=1$ occurs to the left of it.  If $t \geq 2$ and $2 \leq i \leq t$, then let $S$ denote the set of letters occurring to the right of $m_{i-1}$ in $\pi$ and larger than $m_{i}$.

Suppose that the elements of $S$, assuming it is non-empty, are  $s_1<s_2<\cdots<s_r$.  Let $T$ denote the set of letters of $\pi$ belonging to the section $\alpha_i$.  Let $T_1=[s_1-1]\cap T$, $T_j=[s_{j-1}+1,s_j-1]\cap T$ if $2 \leq j \leq r$, and $T_{r+1}=[s_r+1,n]\cap T$ (put $T=T_1$ if $S$ is empty).  Then all letters of $T_{r+1}$ within $\pi$ must occur to the right of all letters of $T_j$ for $j\leq r$ in order to avoid an occurrence of 1-42-3.  To see this, note that the members of $T$ are partitioned into two parts according to whether a letter is greater or less than $s_r$.  So if some member of $T_{r+1}$ occurred to the left of some member of $T_j$ where $j \leq r$, then there would be an occurrence of 1-42-3 in which the ``1'' corresponds to $m_i$ and the ``3'' to $s_r$.  Reasoning inductively, it follows that all of the letters of $T_j$ must occur to the right of all the letters of $T_k$ if $j>k$, that is, $\alpha_i$ may be decomposed as $$\alpha_i=T_1T_2\cdots T_{r+1}.$$   Observe that each $T_j$ must avoid 31-2 since $m_i$ occurs to the left of and is smaller than all letters of $\alpha_i$.  Furthermore, it is seen that no additional restrictions concerning the $T_j$ are necessary, since $\alpha_i$ can be decomposed as described and since the elements of each $T_j$ are either all larger or all smaller than any given member of $S$.

Note that the patterns 31-2 and 13-2 are equivalent, upon replacing $k$ with $n+1-k$ for all $k$.  Let us denote this bijection by $'$.  If $2 \leq i \leq t$, then let $\alpha_i^*$ be defined by
$$\alpha_i^*=T_{r+1}'T_r'\cdots T_1',$$
where the $'$ mapping is applied to the reduced permutation after which the original letters are restored.
Let us now define $f$ by setting
$$f(\pi)=m_t\alpha_t^*m_{t-1}\alpha_{t-1}^*\cdots m_2\alpha_2^* m_1\alpha_1'.$$
It may be verified that $f(\pi) \in \mathcal{S}_n(1\mbox{-}24\mbox{-}3)$ and that $f$ is a bijection that preserves both positions and values of left-to-right minima.
\end{proof}

From \cite{Tenner2013}, we now get that $\Sav{n}{1\da42\da3} = \Sav{n}{25\bar{1}34}$ and $\Sav{n}{1\da24\da3} = \Sav{n}{23\bar{1}54}$.  Then we obtain the following corollary,  which was originally conjectured in Table 2 of \cite{Pudwell2010b}:
\begin{corollary}
$25\bar{1}34 \we 23\bar{1}54$.
\end{corollary}

\subsection{13-4-2 \bf{$\equiv$} 31-4-2}  

In this subsection, we prove the following result.

\begin{theorem}\label{13-4-2}
$13\da4\da2 \equiv 31\da4\da2$. Furthermore, this equivalence respects the last letter statistic.
\end{theorem}
\begin{proof}
Let $a_n$ and $b_n$ denote, respectively, the number of 13-4-2 and 31-4-2 avoiding permutations of length $n$.  We will show $a_n=b_n$ for all $n$, by induction.  Let us assume $n \geq 4$, the $n\leq3$ cases being trivial.  Given $S\subseteq [n-1]$, let $\mathcal{A}_{n,S}$ denote the subset of 13-4-2 avoiding permutations of length $n$ whose set of letters occurring to the \emph{right} of the letter $n$ is $S$.  Likewise, define $\mathcal{B}_{n,S}$ in conjunction with the pattern 31-4-2.  Let $a_{n,S}=|\mathcal{A}_{n,S}|$ and $b_{n,S}=|\mathcal{B}_{n,S}|$.  We will show that $a_{n,S}=b_{n,S}$ for all subsets $S$, which would imply $a_n=b_n$ and complete the induction.

If $S=\emptyset$, then the letter $n$ occurs as the last letter and is easily seen to be extraneous concerning the avoidance of either pattern, implying $a_{n,\emptyset}=a_{n-1}=b_{n-1}=b_{n,\emptyset}$.  So assume $S \neq \emptyset$, and let $a_1<a_2<\cdots<a_i$ denote the elements of $S$.  Let $\pi \in \mathcal{A}_{n,S}$ and $T=[n-1]-S$.  Let $T_1 \subseteq T$ denote the set of all elements of $T$ smaller than $a_1$.  Then all elements of $T_1$ occur after all elements of $T-T_1$ within $\pi$.  To see this, suppose not, and let $a \in T_1$ and $b \in T-T_1$, with $a$ occurring to the left of $b$ within $\pi$.  Then there would exist $c \in T_1$ and $d \in T-T_1$ such that $c$ directly precedes $d$ since $T$ and $T-T_1$ are complementary sets comprising all elements to the left of $n$ within $\pi$.  But then the subsequence $cdna_1$ would be an occurrence of 13-4-2 in $\pi$.

Next, let $T_2\subseteq T-T_1$ denote those elements that are smaller than the element $a_2$.  By the same reasoning, all members of $T_2$ must occur to the right of all members of $T-T_1-T_2$ in order to avoid 13-4-2.  Repeat for the subsequent letters $a_j$ of $S$, $2 \leq j \leq i$, letting $T_j\subseteq T-\cup_{r=1}^{j-1}T_r$ denote the subset whose elements are smaller than $a_j$.  Let $T_{i+1}=T-\cup_{r=1}^iT_r$.  Then $T$ is partitioned as $T=\cup_{r=1}^{i+1}T_r$, with all members of $T_r$ greater than all members of $T_s$ if $1 \leq s <r \leq i+1$.  Furthermore, within $\pi$, all members of $T_r$ occur to the left of all members of $T_s$ if $s<r$, that is, the members of $T$ within $\pi$ may be decomposed as a sequence of permutations $P_{i+1}P_i\cdots P_1$, where $P_j$ is a permutation of the elements of $T_j$.  Thus, $\pi$ may be expressed as \begin{equation}\label{e1}
\pi=P_{i+1}P_i\cdots P_1n\sigma,
\end{equation}
where $\sigma$ is a permutation of the elements of $S$.

Clearly, $\sigma$ and each $P_j$ must avoid 13-4-2 since $\pi$ does.  We now argue that $\sigma$ and each $P_j$ avoiding 13-4-2 is sufficient for $\pi$ to avoid 13-4-2, where $\pi$ is \emph{any} permutation of $[n]$ that is decomposed as in \eqref{e1} and $\sigma$ and the $P_j$ are as defined above.  By an $r$-occurrence of the pattern 13-4-2, we will mean one in which the role of the ``4'' is played by the letter $r$.  Note that $\pi$ contains no $n$-occurrences of 13-4-2, by construction.  Suppose, to the contrary, that $\pi$ does not avoid 13-4-2 and thus contains an $m$-occurrence of 13-4-2 for some $m \in [n-1]$, which we'll denote by $\tau$.  Clearly, the last letter of $\tau$ cannot occur to the left of $n$ within $\pi$ since no $P_j$ contains 13-4-2 and since $r>s$ implies all of the letters in $P_r$ are greater than all letters in $P_s$.  On the other hand, suppose that the ``2'' within $\tau$ is to the right of $n$, while the ``13'' is to the left of $n$.  Then $\tau$ is in fact an $n$-occurrence of 13-4-2 as well, since one could replace the $m$ with $n$.  But this contradicts the fact that $\pi$ contains no $n$-occurrences of 13-4-2.  Since  $n$ clearly cannot be a letter of $\tau$, it follows that $\tau$ must be contained within $\sigma$, a contradiction.  This establishes the claim.

Thus, once the letters of $T$ have been sorted according to their membership in the various $T_j$'s, each $T_j$ and the set $S$ may be ordered according to any permutation that avoids 13-4-2.  The permutations of these sets are then arranged according to \eqref{e1}.  By similar reasoning, if $\lambda \in \mathcal{B}_{n,S}$, then it may be decomposed as
\begin{equation}\label{e2}
\lambda=P_1P_2\cdots P_{i+1}n\sigma,
\end{equation}
where each $P_j$ is a permutation of the elements of $T_j$ and $\sigma$ is a permutation of $S$.  Likewise, $\sigma$ and the $P_j$ avoid 31-4-2, with no further restriction required of these sections.  By the induction hypothesis, the number of choices for $\sigma$ and each $P_j$ in the decompositions for $\pi$ and $\lambda$ in \eqref{e1} and \eqref{e2} are the same, which implies $a_{n,S}=b_{n,S}$ and establishes the equivalence.  The second statement can be obtained from the preceding proof and an induction argument.
\end{proof}

From \cite{Tenner2013}, we get that  $\Sav{n}{13\da4\da2} = \Sav{n}{1\bar{5}342}$ and $\Sav{n}{31\da4\da2}=\Sav{n}{3\bar{5}142}$.  Further, Callan shows $24\da1\da3 \we 3\bar{5}241$ in \cite{Callan2005}.   Combining the above with symmetry yields the following equivalences:
\begin{corollary}
 $1\bar{5}342 \we 3\bar{5}142 \we 3\bar{5}241$.
\end{corollary}
These were first conjectured in \cite{Pudwell2010b}.

\subsection{31-2-4 \bf{$\equiv$} 13-2-4 and 14-2-3 \bf{$\equiv$} 41-2-3}  

A similar argument applies to the two equivalences featured here.

\begin{theorem}\label{31-2-4}
$31\da2\da4 \equiv 13\da2\da4$ and $14\da2\da3 \equiv 41\da2\da3$.
\end{theorem}
\begin{proof}
For the first equivalence, let $\mathcal{A}_n$ and $\mathcal{B}_n$ denote, respectively, the sets consisting of the 31-2-4 and 13-2-4 avoiding permutations of length $n$.  We will define a bijection $f_n$ between $\mathcal{A}_n$ and $\mathcal{B}_n$ in an inductive manner as follows.  For $n \leq 3$, we may clearly take $f_n$ to be the identity, so assume $n \geq 4$.  Let $\alpha=\alpha_1\alpha_2\cdots\alpha_n \in \mathcal{A}_n$.  If $\alpha_n=1$, then let $f_{n}(\alpha)=f_{n-1}(\alpha_1\alpha_2\cdots\alpha_{n-1})1$, so assume $\alpha_n\neq1$.  By a $j$-occurrence of 13-2-4 or 31-2-4, we will mean one in which the role of ``1'' is played by the letter $j$.  Consider whether or not there are any $1$-occurrences of 13-2-4 within $\alpha$.  If there aren't any, then leave the $1$ in its current position and let $\alpha'=\alpha$.  Otherwise, 1 is followed by a letter $a$ such that $1a$ are the first two letters of an occurrence of 13-2-4 in $\alpha$.  We then interchange the positions of 1 and $a$ within $\alpha$ and consider whether or not there is a $1$-occurrence of 13-2-4 in the resulting permutation.  We repeat this process of transposing the letter 1 with its successor until one first reaches the point in which there are no $1$-occurrences of 13-2-4 in the resulting permutation.  Let $\alpha'$ denote in this case the permutation that results from moving $1$ as described.

Note that if $1$ is not moved in the procedure above, then there are no 1-occurrences of 31-2-4 or 13-2-4 in $\alpha'$, while if there is at least one transposition performed, then $\alpha'$ contains no 1-occurrences of 13-2-4, but has 1-occurrences of 31-2-4.  One may verify in either case that $\alpha'$ contains no $j$-occurrences of 31-2-4 for $j>1$.  Let $b$ denote the letter directly following 1 in $\alpha'$.  We now erase the 1 from $\alpha'$ and let $\alpha^*$ denote the resulting permutation of the set $[2,n]$.  Note that $\alpha^*$ contains no occurrences of 31-2-4.  To see this, first note that any occurrence of 31-2-4 in $\alpha^*$ would have to start with the letter directly preceding 1 in $\alpha'$.  Thus, if $\alpha=\alpha'$, then $\alpha^*$ containing a 31-2-4 implies $\alpha$ contains a 1-occurrence of 31-2-4, which it doesn't.  If $\alpha \neq \alpha'$, then $\alpha^*$ containing a 31-2-4 implies $\alpha$ contains a $b$-occurrence of 31-2-4, which is again not possible.

We now apply the bijection $f_{n-1}$ to the permutation $\alpha^*$ (on the letters $\{2,3,\ldots,n\})$.  Let $\beta^*=f_{n-1}(\alpha^*)$.  Insert $1$ directly before the letter $b$ in $\beta^*$ to obtain a permutation $\beta$ of length $n$.  Set $f_n(\alpha)=\beta$.  One may verify that $f_n$ is a bijection once it is shown that inserting $1$ directly before $b$ in $\beta^*$ does not introduce a 1-occurrence of 13-2-4.

To do so, first observe that $\alpha^*$ may be decomposed as $\alpha^*=\rho_1b\rho_2\rho_3$, where the $\rho_i$ are possibly empty with $\rho_2$ containing only letters in $[b+1,n]$ and $\rho_3$ containing only letters in $[b-1]$.  Now apply the algorithm described above to the permutation $\alpha^*$, and then repeatedly to smaller and smaller permutations, until one reaches a  permutation $\lambda$ of $[b+1,n]$.  Suppose $b$ directly preceded $M$ before it was erased in the $b$-th step. Considering $M$ in place of $b$ in $\lambda$, one can apply an induction argument and conclude the following: any letter smaller than $b$ and to the right (left) of it in $\alpha^*$ remains to the right (left) of $b$ in $\beta^*$ and that no letter larger than $b$ and occurring to the left of it in $\alpha^*$ can occur to the right of $b$ in $\beta^*$.  (Note that no letter $d<b$ could be transposed with $b$ in one of the first $b-1$ steps, as there can be no 13-2-4 occurrences starting with $db$ at that point, since none of the letters of $\rho_2$ would have been moved yet and since any letters coming to the right of those in $\rho_2$ at that point would be smaller than $b$.)  Therefore, since $\alpha^*$ can be decomposed as described above, the complete set of letters which lie to the right of $b$ in $\beta^*$ would comprise all of $\rho_3$ and a subset of $\rho_2$.  Note further that all letters of $\rho_3$ will remain to the right of those of $\rho_2$ in $\beta^*$.  Thus, there can be no occurrences of 3-2-4 in $\beta^*$ in which the ``3'' corresponds to $b$.  Inserting 1 directly before $b$ in $\beta^*$ then does not introduce an occurrence of 13-2-4, which implies $f_n$ is a bijection and establishes the first equivalence.

For the second equivalence, one can define a bijection $g_n$ between the sets of 41-2-3 and 14-2-3 avoiding permutations of length $n$ which is analogous to $f_n$ above.  In the first step, one would consider whether or not a permutation $\pi$ that avoids 41-2-3 ends in 1, and if it does not, move the 1 to the right within $\pi$ by a sequence of transpositions of adjacent letters until there are no 1-occurrences of 14-2-3.  After noting the successor of 1 in the resulting permutation, one would erase the 1 and proceed inductively on the new permutation of $[2,n]$.  The remaining steps are similar to those above and the details are left to the reader.
\end{proof}

From \cite{Tenner2013}, we get $\Sav{n}{31\da2\da4} = \Sav{n}{3\bar{5}124}$ and $\Sav{n}{13\da2\da4} = \Sav{n}{1\bar{5}324}$, as well as $\Sav{n}{14\da2\da3} = \Sav{n}{1\bar{3}524}$ and $\Sav{n}{41\da2\da3} = \Sav{n}{5\bar{3}124}$.  Thus the above equivalences imply the following result:
\begin{corollary}
 ${3\bar{5}124}\we{1\bar{5}324}$ and $1\bar{3}524 \we 5\bar{3}124$.
\end{corollary}

\section{Future Work}\label{sec:conc}

We take some time now to outline briefly the current state of the classification of vincular patterns of length 4 according to Wilf-equivalence.

For classical patterns of length 4, it has been shown that there are three Wilf-classes.  B{\'o}na provides a succinct summary of this classification in \cite[pp. 135--136]{Bona2004}.  The class representatives are $1\da2\da3\da4$, $1\da3\da4\da2$, and $1\da3\da2\da4$.

For the consecutive patterns of length 4, the Wilf-classification was resolved by Elizalde and Noy in \cite{Elizalde2003}.  In this case one finds only one equivalence not due to symmetry: $2341 \we 1342$.  The remaining equivalences are due to the reversal and complement operations.

We now move to the Wilf-classification of vincular patterns of length 4 with one or two internal dashes.  We first summarize the above equivalences as they apply to length 4 patterns.

\begin{itemize}
 \item $134\da2 \we 124\da3$
 \item $132\da4 \we 142\da3$
 \item $231\da4 \we 241\da3$
 \item $1\da24\da3 \we 1\da42\da3$
 \item $13\da4\da2 \we 31\da4\da2$
 \item $31\da2\da4 \we 13\da2\da4$
 \item $14\da2\da3 \we 41\da2\da3$
\end{itemize}

Combining these equivalences with those appearing in the literature (listed in the introduction) leaves only two conjectural equivalences, which have each been confirmed to hold for $n\leq 9$.

\begin{conjecture}\label{conj:LastWilf1}
 The following Wilf-equivalences hold:
 \begin{enumerate}
 \item  $23\da1\da4 \we 1\da23\da4$ 
 \item $14\da2\da3 \we 2\da14\da3$.  
 \end{enumerate}
\end{conjecture}

By work in \cite{Baxter2012, Elizalde2006, Kitaev2005}, it has been shown that $1\da23\da4 \we 1\da32\da4 \we 1\da34\da2 \we 1\da43\da2$, but only one representative from this class was given in part (1).   For part (2), it should be noted that work in \cite{Kasraoui2012} proves that $2\da14\da3 \we 2\da41\da3$.  It should also be noted that part (2) is equivalent to $1\bar{3}524 \we 21\bar{3}54$ from the results in \cite{Tenner2013}.

\end{document}